\documentclass[10pt, a4paper]{amsart}
\usepackage[dvips]{epsfig}
\usepackage{amsmath}
\usepackage{amssymb}
\usepackage{amsfonts}
\usepackage{amsthm}
\usepackage{amsbsy}
\usepackage{amsgen}
\usepackage{amscd}
\usepackage{amsopn}
\usepackage{amstext}
\usepackage{amsxtra}
\usepackage{xypic}

\newtheorem{theorem}{Theorem}[section]

\newtheorem{lemma}[theorem]{Lemma}

\newtheorem{defn}[theorem]{Definition}

\theoremstyle{definition}

\begin{document}

\title[A model category structure on the category of simplicial multicategories]
{A model category structure on the category of simplicial
multicategories}
\author[{A. E.} {Stanculescu}]{{Alexandru E.} {Stanculescu}}
\address{\newline Department of Mathematics and Statistics,
\newline Masarykova Univerzita, Kotl\'{a}{\v{r}}sk{\'{a}} 2,\newline
611 37 Brno, Czech Republic}
\email{stanculescu@math.muni.cz}

\begin{abstract}
We establish a Quillen model category structure on the category
of symmetric simplicial multicategories. This model structure extends 
the model structure on simplicial categories due to J. Bergner.
\end{abstract}

\maketitle
\section{Introduction}
A multicategory can be thought of as a generalization of the notion
of category, to the amount that an arrow is allowed to have a source
(or input) consisting of a (possibly empty) string of objects,
whereas the target (or output) remains a single object. Composition
of arrows is performed by inserting the output of an arrow into (one
of) the input(s) of the other. Then a multifunctor is a structure
preserving map between multicategories. For example, every
multicategory has an underlying category obtained by considering
only those arrows with source consisting of strings of length one
(or, one input). At the same time, a multicategory can be thought of
as an ``operad with many objects". The relationship between all 
of these structures can be displayed in a diagram
\[
   \xymatrix{
Monoids  \ar@{^{(}->} [d]  \ar@{^{(}->} [r] & Operads 
\ar@{^{(}->} [d]\\
Categories \ar@{^{(}->} [r] & Multicategories\\
 }
  \]
in which the two composites agree and each horizontal 
(vertical, respectively) inclusion is part of a coreflection 
(reflection, respectively), and $Operads$ stands for 
the category of non-symmetric operads.

By allowing the symmetric groups to act on the various strings of 
objects of a multicategory, and consequently requiring that 
composition of arrows be compatible with these actions in a 
certain natural way, one obtains the concept of symmetric
multicategory. We refer the reader to \cite{Le}, \cite{EM} 
and \cite{BM} for precise definitions, history and examples.

As there is a notion of category enriched over a symmetric 
monoidal category other than the category of sets, the same 
happens with multicategories. We shall mainly consider symmetric
multicategories enriched over simplicial sets, simply called
simplicial multicategories. Similarly, categories enriched over
simplicial sets will be called simplicial categories. The diagram
dispayed above has a simplicially enriched analogue in which 
$Operads$ is now the category of symmetric operads in 
simplicial sets.

In this paper we put a model category structure on the 
category of simplicial multicategories. The notion of weak 
equivalence that we use has been first defined, to the best 
of our knowledge, in \cite[Definition 12.1]{EM}, and it is a 
generalization of the notion of Dwyer-Kan equivalence of simplicial 
categories \cite{DK1}, \cite{Be}. We recall below the definition.

Consider the underlying simplicial category of a simplicial 
multicategory, which is right adjoint to the inclusion
of simplicial categories in simplicial multicategories.
To every simplicial category $\mathrm{C}$ one can 
associate a genuine category $\pi_{0}\mathrm{C}$. 
The objects of $\pi_{0}\mathrm{C}$ are 
the objects of $\mathrm{C}$ and the hom set 
$\pi_{0}\mathrm{C}(x,y)$ is the set of connected 
components of the simplicial set 
$\mathrm{C}(x,y)$. Now, a simplicial multifunctor
$f:\mathrm{M}\rightarrow \mathrm{N}$ is a weak equivalence if
$\pi_{0}f$ is essentially surjective and for every $k\geq 0$ and
every $(k+1)$-tuple $(a_{1},...,a_{k};b)$ of objects of
$\mathrm{M}$, the map $\mathrm{M}_{k}(a_{1},...,a_{k};b)
\rightarrow\mathrm{N}_{k}(f(a_{1}),...,f(a_{k});f(b))$ is a 
weak homotopy equivalence. Our main result is then the following 
\\

\textbf{Theorem.} (Theorem 3.5) \emph{With the class of 
weak equivalences defined above, the category of 
simplicial multicategories admits a cofibrantly generated 
Quillen model category structure.}
\\

To prove this theorem we employ a standard recognition 
principle for cofibrantly generated model categories 
\cite[11.3.1]{Hi}. To be able to apply this principle we 
use the explicit description of a generating set of trivial 
cofibrations of the similar model structure on simplicial 
categories due to J. Bergner \cite{Be}, 
a modification of some parts of Bergner's 
original argument and the model structure for simplicial 
multicategories with fixed set of objects \cite[Theorem 2.1]{BM}. 
The modification is essential for our proof to work, and it shows 
that our argument is not a formal extension of Bergner's.
Still, our line of proof  has a flavour of generality.

Here is the plan of the paper. In section 2 we review 
the notions and results from enriched (multi)category theory 
that we use. We have chosen to work in full generality, in 
the sense that our (symmetric multi)categories are enriched 
over an arbitrary closed symmetric monoidal category. 
This choice does not complicate things. The proof of the 
main result is presented in section 3. The modification 
alluded to above is contained in the proof of Lemma 3.6.

\section{Review of enriched categories and symmetric multicategories} 

Throughout this section $\mathcal{V}$ is a complete and 
cocomplete closed symmetric monoidal category with unit $I$ 
and initial object $\emptyset$. 

\subsection{Enriched categories}
The small $\mathcal{V}$-categories together with the $\mathcal{V}$-functors
between them form a category written $\mathcal{V}\text{-}{\bf Cat}$. 
If $S$ is a set, we denote by $\mathcal{V}$\text{-}{\bf Cat}$(S)$
the category of small $\mathcal{V}$-categories with fixed set of objects $S$.
We denote by $\mathcal{I}$ the $\mathcal{V}$-category with a 
single object $\ast$ and $\mathcal{I}(\ast,\ast)=I$. We
denote by $Ob$ the functor sending a $\mathcal{V}$-category
to its set of objects. 

If $\mathcal{K}$ is a class of maps of $\mathcal{V}$, we say that a 
$\mathcal{V}$-functor $f:\mathcal{A}\rightarrow \mathcal{B}$ is 
{\bf locally in} $\mathcal{K}$ if for each pair $x,y\in \mathcal{A}$ of 
objects, the map $f_{x,y}:\mathcal{A}(x,y)\rightarrow \mathcal{B}(f(x),f(y))$ 
is in $\mathcal{K}$. When $\mathcal{K}$ is the class of isomorphisms
of $\mathcal{V}$, a $\mathcal{V}$-functor which is locally an
isomorphism is called {\bf full and faithful}.

Let $f:\mathcal{A}\rightarrow \mathcal{B}$ be a
$\mathcal{V}$-functor and let $u=Ob(f)$. Then $f$ factors as
$\mathcal{A}\overset{f^{u}} \rightarrow u^{\ast}\mathcal{B}
\rightarrow \mathcal{B}$, where $f^{u}$ is a map in
$\mathcal{V}$\text{-}{\bf Cat}$(Ob(\mathcal{A}))$. One has
$u^{\ast}\mathcal{B}(a,a')=\mathcal{B}(f(a),f(a'))$ and
$u^{\ast}\mathcal{B} \rightarrow \mathcal{B}$ is full and faithful.

For an object $X$ of $\mathcal{V}$, we denote by 
$2_{X}$ the $\mathcal{V}$-category with two objects 0 and 1, and with 
$2_{X}(0,0)=2_{X}(1,1)=I$, $2_{X}(0,1)=X$ and $2_{X}(1,0)=\emptyset$.

Let $\mathcal{V}^{\Delta^{op}}$ have the levelwise monoidal product. 
We have a full and faithful functor
$$\varphi:\mathcal{V}^{\Delta^{op}}\text{-}{\bf Cat}\longrightarrow
(\mathcal{V}\text{-}{\bf Cat})^{\Delta^{op}}$$ given by
$Ob(\varphi(\mathcal{A})([n]))=Ob(\mathcal{A})$ for all $[n]\in
\Delta$ and $\varphi(\mathcal{A})([n])(a,b)=\mathcal{A}(a,b)([n])$
for all $a,b \in Ob(\mathcal{A})$. The category
$\mathcal{V}^{\Delta^{op}}\text{-}{\bf Cat}$ is coreflective in
$(\mathcal{V}\text{-}{\bf Cat})^{\Delta^{op}}$, that is, the functor
$\varphi$ has a right adjoint $R$, defined as follows. For
$\mathcal{B}\in (\mathcal{V}\text{-}{\bf Cat})^{\Delta^{op}}$ 
we put $Ob(R(\mathcal{B}))=\lim Ob(\mathcal{B}([n]))$ and 
$R(\mathcal{B})((a_{n}),(b_{n}))([m])=\mathcal{B}([m])(a_{m},b_{m})$.

\subsection{Enriched symmetric multicategories}
For the notions of {\bf symmetric} $\mathcal{V}$-{\bf multicategory} 
and {\bf symmetric} $\mathcal{V}$-{\bf multifunctor} we refer 
the reader to \cite[Definition 2.2.21]{Le} or \cite[2.1, 2.2]{EM}. 

If $\mathrm{M}$ is a symmetric $\mathcal{V}$-multicategory, 
$k\geq 0$ is an integer and $(a_{1},...,a_{k};b)$ is a $(k+1)$-tuple of 
objects, we follow \cite[2.1(2)]{EM} and denote by 
$\mathrm{M}_{k}(a_{1},...,a_{k};b)$ the $\mathcal{V}$-object of
``$k$-morphisms". When $k=0$, the $\mathcal{V}$-object of 
$0$-morphisms is denoted by $\mathrm{M}(\ ;b)$.

The small symmetric $\mathcal{V}$-multicategories together 
with the symmetric $\mathcal{V}$-multifunctors between 
them form a category written $\mathcal{V}$\text{-}{\bf SymMulticat}.
When $\mathcal{V}$ is the category $Set$ of sets, symmetric
$Set$-multicategories will be simply referred to as
{\bf multicategories}, and the category will be 
denoted by $\mathbf{SymMulticat}$.

A {\bf symmetric} $\mathcal{V}$-{\bf multigraph} is by definition a
symmetric $\mathcal{V}$-multicategory without composition and unit
maps. We shall write $\mathcal{V}$\text{-}{\bf SymMultigraph} for the
category of symmetric $\mathcal{V}$-multigraphs with the evident
notion of arrow. When $\mathcal{V}=Set$, the category is denoted by
{\bf SymMultigraph}.

We denote by $Ob$ the functor sending a
symmetric $\mathcal{V}$-multicategory (or a symmetric
$\mathcal{V}$-multigraph) to its set of objects. The functor $Ob$ is
a Grothendieck bifibration. There is a free-forgetful adjunction
\[
 \xymatrix{
\mathcal{F}:\mathcal{V}\text{-}{\bf SymMultigraph} 
\ar @<2pt> [rr] \ar[dr]_{Ob} & & 
\mathcal{V}\text{-}{\bf SymMulticat}:\mathcal{U} 
\ar @<2pt> [ll] \ar[dl]^{Ob} \qquad (1)\\
& Set\\
}
  \]
We write $\mathcal{V}$\text{-}{\bf SymMultigraph}$(S)$
(resp. $\mathcal{V}$\text{-}{\bf SymMulticat}$(S)$) for 
the fibre category over a set $S$. The category 
$\mathcal{V}$\text{-}{\bf SymMultigraph}$(S)$
admits a nonsymmetric monoidal product which 
preserves filtered colimits in each variable, and the 
category $\mathcal{V}$\text{-}{\bf SymMulticat}$(S)$
is precisely the category of monoids in 
$\mathcal{V}$\text{-}{\bf SymMultigraph}$(S)$
with respect to this monoidal product \cite[Section 7]{BM}.
From the general theory of limits and colimits in bifibrations it 
follows that $\mathcal{V}$\text{-}{\bf SymMulticat} is 
complete and cocomplete. If $\mathcal{V}$ is moreover accessible, 
it follows from the general theory \cite[Theorem 5.3.4]{MP} 
that $\mathcal{V}$\text{-}{\bf SymMulticat} is accessible.

$\mathcal{V}$-categories and symmetric
$\mathcal{V}$-multicategories can be related by the 
adjunction
\[
   \xymatrix{
E:\mathcal{V}\text{-}{\bf Cat} \ar @<2pt> [rr] \ar[dr]_{Ob} 
& & \mathcal{V}\text{-}{\bf SymMulticat}:(-)_{1} 
\ar @<2pt> [ll] \ar[dl]^{Ob} \qquad (2)\\
   & Set\\
}
  \]
where
$$(E\mathcal{A})_{k}(a_{1},...,a_{k};b)=
\begin{cases}
\mathcal{A}(a_{1},b) & \text{if } k=1,\\
\emptyset & \text{otherwise},\\
\end{cases} $$
and $\mathrm{M}_{1}(a,b)=\mathrm{M}_{1}(a;b)$. 
$(-)_{1}$ is referred to as the underlying 
$\mathcal{V}$-category functor. The functor $E$ is
full and faithful.

Let $\mathcal{K}$ be a class of maps of $\mathcal{V}$.
We say that a symmetric $\mathcal{V}$-multifunctor
$f:\mathrm{M}\rightarrow \mathrm{N}$ is {\bf locally in}
$\mathcal{K}$ if for each integer $k\geq 0$ and each 
$(k+1)$-tuple of objects $(a_{1},...,a_{k};b)$, the map
$f:\mathrm{M}_{k}(a_{1},...,a_{k};b)\rightarrow
\mathrm{N}_{k}(f(a_{1}),...,f(a_{k});f(b))$ is in $\mathcal{K}$.
When $\mathcal{K}$ is the class of isomorphisms of $\mathcal{V}$, 
a symmetric $\mathcal{V}$-multifunctor which is locally an 
isomorphism is called {\bf full and faithful}.

We recall that a $\mathcal{V}$-{\bf multigraph}
$\mathrm{M}$ consists of a set of objects $Ob(\mathrm{M})$ together
with an object $\mathrm{M}_{k}(a_{1},...,a_{k};b)$ of $\mathcal{V}$
assigned to each integer $k\geq 0$ and each $(k+1)$-tuple of
objects$(a_{1},...,a_{k};b)$. We write
$\mathcal{V}$\text{-}{\bf Multigraph} for the resulting category. 
In the case when $\mathcal{V}=Set$, this category is denoted by
{\bf Multigraph} and its objects will be called
{\bf multigraphs}.

The forgetful functor from symmetric $\mathcal{V}$-multigraphs to
$\mathcal{V}$-multigraphs has a left adjoint $Sym$ defined by
$$(Sym\mathrm{M})_{k}(a_{1},...,a_{k};b)=\underset{\sigma \in \Sigma_{k}}
\coprod \mathrm{M}_{k}(a_{\sigma^{-1}(1)},...,a_{\sigma^{-1}(k)};b)$$ 
where $\Sigma_{k}$ is the symmetric group on $k$ elements.

For each integer $k\geq 0$ we denote by
$\underline{k+1}$ the set $\{1,2,...,k,\ast\}$, where $\ast \not \in
\{1,2,...,k\}$. We have a functor $$(\underline{k+1},\ ):
\mathcal{V}\rightarrow \mathcal{V}\text{-}{\bf Multigraph}$$ given by
$(\underline{k+1},A)_{n}(a_{1},...,a_{n};b)=\emptyset$ unless $n=k$
and $a_{i}=i$ and $b=\ast$, in which case we define it to be $A$. To
give a map of $\mathcal{V}$-multigraphs
$(\underline{k+1},A)\rightarrow \mathrm{M}$ is to give a map
$A\rightarrow \mathrm{M}_{k}(a_{1},...,a_{k};b)$.

Let $\mathcal{V}^{\Delta^{op}}$ have the levelwise
monoidal product. We have a full and faithful functor
$$\varphi':\mathcal{V}^{\Delta^{op}}\text{-}{\bf SymMulticat}
\longrightarrow (\mathcal{V}\text{-}{\bf SymMulticat})^{\Delta^{op}}$$ 
given by $Ob(\varphi'(\mathrm{M})([n]))=Ob(\mathrm{M})$ for all $[n]\in
\Delta$ and $\varphi'(\mathrm{M})([n])_{k}(a_{1},...,a_{k};b)=
\mathrm{M}_{k}(a_{1},...,a_{k};b)([n])$
for each $(k+1)$-tuple of objects $(a_{1},...,a_{k};b)$. 
Clearly, a map $f$ in $\mathcal{V}^{\Delta^{op}}$\text{-}{\bf SymMulticat}
is full and faithful if and only if for each $[n]\in \Delta$,
$\varphi'(f)([n])$ is full and faithful in $\mathcal{V}$\text{-}{\bf SymMulticat}.
The category $\mathcal{V}^{\Delta^{op}}\text{-}{\bf SymMulticat}$ is
coreflective in $(\mathcal{V}\text{-}{\bf SymMulticat})^{\Delta^{op}}$.
The right adjoint to $\varphi'$, which we denote by $R'$,
is defined as follows. For 
$\mathrm{M}\in (\mathcal{V}\text{-}{\bf SymMulticat})^{\Delta^{op}}$ 
we put $Ob(R'(\mathrm{M}))=\lim Ob(\mathrm{M}([n]))$ and
$$R'(\mathrm{M})_{k}((a_{n}^{1}),...,(a_{n}^{k});(b_{n}))([m])
=\mathrm{M}([m])_{k}(a_{m}^{1},...,a_{m}^{k};b_{m})$$ We have a
commutative square of adjunctions
\[
   \xymatrix{
\mathcal{V}^{\Delta^{op}}\text{-}{\bf Cat} \ar @<2pt> [rr]^{E} \ar @<2pt>
[d]^{\varphi} & &\mathcal{V}^{\Delta^{op}}\text{-}{\bf SymMulticat} \ar
@<2pt> [ll]^{(\_)_{1}}
 \ar @<2pt> [d]^{\varphi'} \qquad (3)\\
(\mathcal{V}\text{-}{\bf Cat})^{\Delta^{op}} \ar @<2pt>
[rr]^{E^{\Delta^{op}}} \ar @<2pt> [u]^{R} & &
(\mathcal{V}\text{-}{\bf SymMulticat})^{\Delta^{op}}
 \ar @<2pt> [ll]^{(\_)_{1}^{\Delta^{op}}} \ar @<2pt> [u]^{R'}\\
  }
  \]

\section{The Dwyer-Kan model structure on {\bf S}\text{-}{\bf SymMulticat}}

We denote by {\bf Cat} the category of small categories. We say that an
arrow $f:C\rightarrow D$ of {\bf Cat} is an {\bf isofibration} if for any $x\in Ob(C)$
and any isomorphism $v:y'\rightarrow f(x)$ in $D$, there exists an isomorphism
$u:x'\rightarrow x$ in $C$ such that $f(u)=v$. 

We denote by {\bf S} the category of simplicial sets, regarded as 
having the Quillen model structure. Let $\pi_{0}:{\bf S}\rightarrow Set$ be
the set of connected components functor. By change of base it induces a functor
$\pi_{0}:{\bf S}\text{-}{\bf Cat} \rightarrow {\bf Cat}$ which is the
identity on objects.

\begin{defn} \cite{Be} Let $f:\mathcal{A}\rightarrow \mathcal{B}$
be a morphism in {\bf S}\text{-}{\bf Cat}.

1. The morphism $f$ is a $\mathrm{DK}$-{\bf equivalence} if 
$f$ is locally a weak homotopy equivalence and $\pi_{0}f$ is
essentially surjective.

2. The morphism $f$ is a $\mathrm{DK}$-{\bf fibration} if $f$
is locally a Kan fibration and $\pi_{0}f$ is an isofibration.

3. The morphism $f$ is a {\bf trivial fibration} if $f$ is a
DK-equivalence and a DK-fibration.
\end{defn}
A morphism is a trivial fibration if and only if it is
surjective on objects and locally a trivial fibration.

Let $S$ be a set. We recall \cite[Section 7]{DK2}
that the category {\bf S}\text{-}{\bf Cat}$(S)$ has a
model structure in which the weak equivalences
and the fibrations are the simplicial functors which are
locally a weak homotopy equivalence and a Kan fibration, 
respectively.
\begin{theorem} \cite{Be} The category {\bf S}\text{-}{\bf Cat} 
of simplicial categories admits a cofibrantly generated model 
structure in which the weak equivalences are the
DK-equivalences and the fibrations are the DK-fibrations. A
generating set of trivial cofibrations consists of

(B1) $\{2_{X}\overset{2_{j}}\longrightarrow
2_{Y}\}$, where $j$ is a horn inclusion, and

(B2) inclusions $\mathcal{I} \overset{\delta_{y}} \rightarrow
\mathcal{H}$, where $\{\mathcal{H}\}$ is a set of representatives
for the isomorphism classes of simplicial categories on two objects
which have  countably many simplices in each function complex.
Furthermore, each such $\mathcal{H}$ is required to be cofibrant 
and weakly contractible in {\bf S}\text{-}{\bf Cat}$(\{x,y\})$. 
Here $\{x,y\}$ is the set with elements $x$ and $y$ and 
$\delta_{y}$ omits $y$. 
\end{theorem}
Recall from 2.2, adjunction (2), the functor 
$(-)_{1}:{\bf S}\text{-}{\bf SymMulticat}
\rightarrow {\bf S}\text{-}{\bf Cat}$.
\begin{defn} Let $f:\mathrm{M}\rightarrow \mathrm{N}$ be a
morphism in {\bf S}\text{-}{\bf SymMulticat}.

1. \cite[Definition 12.1]{EM} The morphism $f$ is a 
{\bf weak equivalence} if $f$ is locally a weak homotopy 
equivalence and $\pi_{0}f_{1}$ is essentially surjective.

2. The morphism $f$ is a {\bf fibration} if $f$ is locally a 
Kan fibration and $\pi_{0}f_{1}$ is an isofibration.

3. The morphism $f$ is a {\bf trivial fibration} if $f$ is a
weak equivalence and a fibration.
\end{defn}

Recall from 2.2, adjunction (2), the functor 
$E:{\bf S}\text{-}{\bf Cat} \rightarrow 
{\bf S}\text{-}{\bf SymMulticat}$.
\begin{lemma}
A morphism $f:\mathrm{M}\rightarrow \mathrm{N}$
in {\bf S}\text{-}{\bf SymMulticat} is a 

(1) weak equivalence if and only if $f$ is locally a weak 
homotopy equivalence and $f_{1}$ is a weak equivalence in 
{\bf S}\text{-}{\bf Cat};

(2) fibration if and only if $f$ is locally a Kan fibration and 
$f_{1}$ is a fibration in {\bf S}\text{-}{\bf Cat};

(3) trivial fibration if and only if $f$ is locally a trivial fibration
and $f_{1}$ is a trivial fibration in {\bf S}\text{-}{\bf Cat}.

It follows easily that $(-)_{1}$ and $E$ both preserve
weak equivalences and fibrations.
\end{lemma}
Our main result is
\begin{theorem} The category {\bf S}\text{-}{\bf SymMulticat} 
admits a cofibrantly generated model category
structure with weak equivalences and fibrations as in the 
previous definition. The model structure is right proper.
\end{theorem}
\begin{proof} 
The proof will depend on Lemmas 3.6, 3.7, 3.8 and 3.9
which appear below. We take as the set of generating cofibrations
the set $\mathrm{I}$ consisting of $E(\emptyset \rightarrow \mathcal{I})
\cup \{\mathcal{F}Sym(\underline{k+1},i)\}_{k\geq 0}$, where
$i$ is a generating cofibration of $\mathbf{S}$. We take as the 
set of generating trivial cofibrations the set $\mathrm{J}$ 
consisting of $E(B2) \cup \{\mathcal{F}Sym(\underline{k+1},j)\}_{k\geq 0}$, 
where $j$ is a horn inclusion. Writing $\mathrm{W}$ for the 
class of weak equivalences, we see by Lemma 3.4 that 
$\mathrm{W}\cap \mathrm{J}\text{-}inj=\mathrm{I}\text{-}inj$.
By \cite[11.3.1]{Hi} it will suffice to show that 
$\mathrm{J}$\text{-}$cof\subset \mathrm{W}$.

The composite forgetful functor from 
${\bf S}\text{-}{\bf SymMulticat}$ to 
${\bf S}\text{-}{\bf Multigraph}$ preserves filtered 
colimits. This can be seen, for example, by adapting the 
proof of the corresponding fact for enriched categories 
\cite[Corollary 3.4]{KL}. Since a transfinite composition 
of weak homotopy equivalences is a weak homotopy equivalence, 
the next lemma completes the proof of the existence of 
the model structure. Right properness is standard, 
see for example \cite[Proposition 3.5]{Be}.
\end{proof}
For the interested reader, the class of cofibrations of 
the model structure constructed in Theorem 3.5
can be given an explicit description \cite{St}.
Since we shall not need this description, we won't
go into details.
\begin{lemma} 
Let $\delta_{y}:\mathcal{I}\rightarrow \mathcal{H}$ be
a map belonging to the set B2 from Theorem 3.2.
Then in the pushout diagram
\[
\xymatrix{
E\mathcal{I} \ar[r]^{x} \ar[d]_{E\delta_{y}} & \
\mathrm{M} \ar[d]\\
E\mathcal{H} \ar[r] & \mathrm{N}\\
}
  \]
the map $\mathrm{M}\rightarrow \mathrm{N}$ is a 
weak equivalence.
\end{lemma}
\begin{proof}
We begin with a remark. For every set $S$
the category {\bf S}\text{-}{\bf SymMulticat}$(S)$ 
admits a model structure in which the weak 
equivalences and fibrations are defined locally
\cite[Theorem 2.1]{BM}. The adjunction (2) 
from 2.2 restricts to a Quillen pair
$$E:{\bf S}\text{-}{\bf Cat}(S)\rightleftarrows 
{\bf S}\text{-}{\bf SymMulticat}(S):(-)_{1}$$
We now factor (2.1) the map $\delta_{y}$ as $\mathcal{I}
\overset{(\delta_{y})^{u}} \longrightarrow u^{*}\mathcal{H}
\rightarrow \mathcal{H}$ where $u=Ob(\delta_{y})$ and 
then we take consecutive pushouts:
\[
\xymatrix{
E\mathcal{I} \ar[r]^{x} \ar[d]_{E(\delta_{y})^{u}} & 
\mathrm{M} \ar[d]^{j}\\
Eu^{*}\mathcal{H} \ar[d] \ar[r] & \mathrm{M'} \ar[d]\\
E\mathcal{H} \ar[r] & \mathrm{N}\\
}
  \]
By Lemma 3.7 the map $(\delta_{y})^{u}$ is a 
trivial cofibration in the category of simplicial monoids, 
therefore the map $j$ is a trivial cofibration in 
{\bf S}\text{-}{\bf SymMulticat}$(Ob(\mathrm{M}))$. 
We claim that $\mathrm{M'}\rightarrow \mathrm{N}$ is a full
and faithful inclusion. For, apply the functor $\varphi'$ from
2.2 to the bottom pushout diagram above. 
Using adjunction (3) from 2.2 we obtain a pushout diagram
\[
\xymatrix{
E^{\Delta^{op}}\varphi(u^{*}\mathcal{H}) \ar[d] \ar[r] & 
\varphi'\mathrm{M'} \ar[d]\\
E^{\Delta^{op}}\varphi(\mathcal{H}) \ar[r] & \varphi'\mathrm{N}\\
}
  \]
in {\bf SymMulticat}$^{\Delta^{op}}$. Evaluating at 
$[n]\in \Delta$ and applying Lemma 3.8 to the resulting
pushout diagram we obtain the claim.

Therefore, the map $\mathrm{M}\rightarrow \mathrm{N}$
is a locally a weak homotopy equivalence, and an easy
diagram chase shows that it is actually a weak equivalence.
\end{proof}
\begin{lemma}
Let $\mathcal{A}$ be a cofibrant simplicial category. 
Then for each $a\in Ob(\mathcal{A})$ the simplicial 
monoid $a^{\ast}\mathcal{A}=\mathcal{A}(a,a)$
is cofibrant.
\end{lemma}
\begin{proof}
Let $S=Ob(\mathcal{A})$. $\mathcal{A}$ is cofibrant if and only if it is
cofibrant as an object of {\bf S}\text{-}{\bf Cat}$(S)$. The cofibrant objects
of {\bf S}\text{-}{\bf Cat}$(S)$ are characterized in \cite[7.6]{DK2}: they
are the retracts of free simplicial categories. Therefore it
suffices to prove that if $\mathcal{A}$ is a free simplicial
category then $a^{\ast}\mathcal{A}$ is a free simplicial category
for all $a\in S$. Recall \cite[4.5]{DK2} that $\mathcal{A}$ is a
free simplicial category if and only if $(i)$ for all $n\geq 0$ the category
$\varphi(\mathcal{A})_{n}$ (see section 2 for the functor $\varphi$) 
is a free category on a graph $G_{n}$, and $(ii)$ for all 
epimorphisms $\alpha:[m]\rightarrow [n]$
of $\Delta$, $\alpha^{\ast}:\varphi(\mathcal{A})_{n}\rightarrow
\varphi(\mathcal{A})_{m}$ maps $G_{n}$ into $G_{m}$.

Let $a\in S$. The category $\varphi(a^{\ast}\mathcal{A})_{n}$ is a
full subcategory of $\varphi(\mathcal{A})_{n}$ with object set
$\{a\}$, and we will show that it is free as well. A set
$G^{a^{\ast}\mathcal{A}}_{n}$ of generators can be described as
follows. An element of $G^{a^{\ast}\mathcal{A}}_{n}$ is a path 
from $a$ to $a$ in $\varphi(\mathcal{A})_{n}$ such that every 
arrow in the path belongs to $G_{n}$ and $a$ does not 
appear anywhere else in the path. Since every epimorphism 
$\alpha:[m]\rightarrow [n]$ of $\Delta$ has a section, 
$\alpha^{\ast}$ maps $G^{a^{\ast}\mathcal{A}}_{n}$ 
into $G^{a^{\ast}\mathcal{A}}_{m}$.
\end{proof}
\begin{lemma}
Let $A$ and $B$ be two small categories and let $i:A\rightarrow
B$ be a full and faithful inclusion. Let $\mathrm{M}$ be a
multicategory. Then in a pushout diagram of the form
\[
\xymatrix{
EA \ar[d]_{Ei} \ar[r] & \mathrm{M} \ar[d]\\
EB \ar[r] & \mathrm{N}\\
}
  \]
the map $\mathrm{M}\rightarrow \mathrm{N}$ is a full 
and faithful inclusion.
\end{lemma}
\begin{proof}
Let $(B-A)^{+}$ be the preorder with objects all finite subsets
$S\subset Ob(B)-Ob(A)$, ordered by inclusion. For $S\in
(B-A)^{+}$, let $A_{S}$ be the full subcategory of $B$ with objects
$Ob(A)\cup S$. Then $B=colim_{(B-A)^{+}} A_{S}$. On the other
hand, a filtered colimit of full and faithful inclusions of
multicategories is a full and faithful inclusion. This is because
the forgetful functor from {\bf SymMulticat} to
{\bf SymMultigraph} preserves filtered colimits (as can 
be seen by adapting the proof of the corresponding 
fact for enriched categories \cite[Corollary 3.4]{KL}, for example)
and a filtered colimit of full and faithful inclusions of multigraphs is 
a full and faithful inclusion. Therefore one can assume 
that $Ob(B)=Ob(A)\cup \{q\}$, where $q\not \in Ob(A)$. 

The pushout in the statement of the lemma is the composite pushout
\[
\xymatrix{
EA \ar[d]_{Ei} \ar[r] & E\mathrm{M}_{1}
\ar[d] \ar[r]^{\epsilon_{M}} & \mathrm{M} \ar[d]\\
EB \ar[r] & EC \ar[r] & \mathrm{N}\\
}
  \]
where the first square on the left is a pushout in 
{\bf Cat} before applying the functor $E$ and
$\epsilon_{\mathrm{M}}$ is the counit of the adjunction (2)
from 2.2 (with $\mathcal{V}=Set$). Recall now that the 
pushout of a full and faithful inclusion of categories along
any functor is a full and faithful inclusion \cite[Proposition 5.2]{FL}.
Therefore it is enough to consider the following situation:
$\mathrm{M}$ is a multicategory, 
$i:\mathrm{M}_{1}\rightarrow B$ 
is a full and faithful inclusion with 
$Ob(B)=Ob(\mathrm{M})\sqcup \{q\}$
and the pushout diagram is
\[
\xymatrix{
E\mathrm{M}_{1} \ar[d]_{Ei} \ar[r]^{\epsilon_{M}} & 
\mathrm{M} \ar[d]\\
EB \ar[r] & \mathrm{N}\\
}
  \]
The fact that $\mathrm{M}\rightarrow \mathrm{N}$ 
is a full and faithful inclusion follows then by taking 
$\mathcal{V}=Set$ in the next lemma.
\end{proof}
\begin{lemma} Let $\mathcal{V}$ be a cocomplete 
closed symmetric monoidal category. 
Let $\mathrm{M}$ be a small symmetric 
$\mathcal{V}$-multicategory, $\mathrm{M}_{1}$ its 
underlying $\mathcal{V}$ category, 
$B$ a small $\mathcal{V}$-category with
$Ob(B)=Ob(\mathrm{M})\sqcup \{q\}$ and 
$i:\mathrm{M}_{1}\rightarrow B$ a full and faithful 
inclusion. Then in a pushout diagram of the form
\[
\xymatrix{
E\mathrm{M}_{1} \ar[d]_{Ei} \ar[r]^{\epsilon_{M}} 
& \mathrm{M} \ar[d]\\
EB \ar[r] & \mathrm{N}\\
}
  \]
the map $\mathrm{M}\rightarrow \mathrm{N}$ is a full and faithful
inclusion. Here $\epsilon_{\mathrm{M}}$ is the counit of the
adjunction (2) from 2.2.
\end{lemma}
\begin{proof} Let $\otimes$ be the monoidal product of $\mathcal{V}$.
We shall explicitly describe the $\mathcal{V}$-objects of
$k$-morphisms of $\mathrm{N}$. For $k\geq 0$ and
$(a_{1},...,a_{k};a)$ a $(k+1)$-tuple of objects with $a\in
Ob(\mathrm{M})$ and $a_{i}\in Ob(\mathrm{M})$ ($i=1,...,k$), we put
$\mathrm{N}_{k}(a_{1},...,a_{k};a)=\mathrm{M}_{k}(a_{1},...,a_{k};a)$.
Then we set $\mathrm{N}(\ ;q)= \int^{x\in Ob(\mathrm{M})}
B(x,q)\otimes M(\ ;x)$ and
$$\mathrm{N}_{k}(a_{1},...,a_{k};q)=\int^{x\in Ob(\mathrm{M})}
B(x,q)\otimes M_{k}(a_{1},...,a_{k};x)$$
if $a_{i}\in Ob(\mathrm{M})$ ($i=1,...,k$).
Next, let $(a_{1},...,a_{k})$ be a $k$-tuple of objects of
$\mathrm{M}$. For each $1\leq s\leq k$ let $\{i_{1},...,i_{s}\}$ be
a nonempty subset of $\{1,...,k\}$. We denote by
$(a_{1},...,a_{k})^{q_{i_{1},...,i_{s}}}$ the $k$-tuple of objects
of $B$ obtained by inserting $q$ in the $k$-tuple
$(a_{1},...,a_{k})$ at the spot $i_{j}$ ($1\leq j\leq s$). For each
$1\leq j\leq s$ and $x_{i_{j}}\in Ob(\mathrm{M})$ we denote by
$(a_{1},...,a_{k})^{\{x_{i_{1}},...,x_{i_{s}}\}}$ the $k$-tuple of
objects of $\mathrm{M}$ obtained by inserting $x_{i_{j}}$ in the
$k$-tuple $(a_{1},...,a_{k})$ at the spot $i_{j}$. We put
$$\mathrm{N}_{k}((a_{1},...,a_{k})^{q_{i_{1},...,i_{s}}};a)=$$
$$=\int^{x_1\in Ob(\mathrm{M})}...\int^{x_s\in Ob(\mathrm{M})}
\mathrm{M}_{k}((a_{1},...,a_{k})^{\{x_{i_{1}},...,x_{i_{s}}\}};a)\otimes
B(q,x_{i_{1}})\otimes...\otimes B(q,x_{i_{s}})$$ if $a\in
Ob(\mathrm{M})$, and
$$\mathrm{N}_{k}((a_{1},...,a_{k})^{q_{i_{1},...,i_{s}}};q)=$$
$=\int^{x\in Ob(\mathrm{M})}\int^{x_{i_{1}}\in
Ob(\mathrm{M})}...\int^{x_{i_{s}}\in Ob(\mathrm{M})} B(x,q)\otimes
\mathrm{M}_{k}((a_{1},...,a_{k})^{\{x_{i_{1}},...,x_{i_{s}}\}};x)\otimes
B(q,x_{i_{1}})\otimes...\otimes B(q,x_{i_{s}})$

This completes the definition of the $\mathcal{V}$-objects of
$k$-morphisms of $\mathrm{N}$. To prove that $\mathrm{N}$ is a
symmetric $\mathcal{V}$-multicategory is long and tedious. Once this
is proved, the fact that it has the desired universal property
follows.
\end{proof}

\end{document}